\documentclass{amsart}
\usepackage{amsmath}
\usepackage{amssymb}
\usepackage{amsthm,amssymb,amsmath,enumerate,graphicx}

\newtheorem{theorem}{Theorem}[section]
\newtheorem{lemma}[theorem]{Lemma}
\newtheorem{assertion}[theorem]{Assertion}

\newtheorem{observation}[theorem]{Observation}
\newtheorem{corollary}[theorem]{Corollary}

\newtheorem{conjecture}[theorem]{Conjecture}

\theoremstyle{remark}
\newtheorem{notation}[theorem]{Notation}
\newtheorem{remark}[theorem]{Remark}
\newtheorem{definition}[theorem]{Definition}
\newtheorem{example}[theorem]{Example}

\newcommand{\cf}{\mathcal F}

\newcommand{\ch}{\mathcal H}

\begin{document}

\makeatletter

\makeatother
\author{Ron Aharoni}
\address{Department of Mathematics\\ Technion}
\email[Ron Aharoni]{raharoni@gmail.com}
\thanks{\noindent The research of the first author was
supported by BSF grant no. $2006099$, by GIF grant no. I
$-879-124.6/2005$, by the Technion's research promotion fund, and by
the Discont Bank chair.}

\author{David Howard}
\address{Department of Mathematics\\ Colgate University}
\email[David Howard]{dmhoward@colgate.edu}
\thanks{\noindent The research of the second author was
supported by BSF grant no. $2006099$, and by ISF grants Nos.
$779/08$, $859/08$ and $938/06$.
 }

\begin{abstract}
Two hypergraphs $H_1,\ H_2$ are called  {\em cross-intersecting} if
$e_1 \cap e_2 \neq \emptyset$ for every pair of edges $e_1 \in
H_1,~e_2 \in H_2$. Each of the hypergraphs is  then said to {\em
block} the other. Given parameters $n,r,m$ we determine the maximal
size of a sub-hypergraph of $[n]^r$ (meaning that it is $r$-partite,
with all sides of size $n$)
  for which there exists a blocking
  sub-hypergraph of $[n]^r$ of size $m$. The answer involves a fractal-like (that is, self-similar) sequence, first studied by Knuth.
We also study the same question with $\binom{n}{r}$ replacing $[n]^r$.
 \end{abstract}

\title{Cross-intersecting pairs of hypergraphs}
\maketitle
\section{Blockers in $r$-partite hypergraphs}
\subsection{Blockers}\label{kk}
For a set $A$ and a number $r$ let $\binom{A}{r}$ be the set of all subsets of size $r$ of $A$. Given numbers $r$ and $n$ let $[n]=\{1, 2,
\ldots,n\}$,   and let $[n]^r$ be the complete
$r$-partite hypergraph with all sides being equal to $[n]$. Let $U$
be either $\binom{[n]}{r}$ or $[n]^r$, and let $F$ be a
sub-hypergraph of $U$. The {\em blocker} $B(F)$  of  $F$  is the set
of those edges of $U$ that meet all edges of $F$. For a number $t$
we denote by $b(t)$ the maximal size of $|B(F)|$, where $F$ ranges
over all sets of $t$ edges in $U$. Which of the two meanings of
``$b(t)$'' we are using will be clear from the context.

\subsection{Background - the Erd\H{o}s-Ko-Rado theorem and rainbow matchings}

A {\em matching} is a collection of disjoint sets.
 The largest size of a matching in a hypergraph
$H$ is denoted by $\nu(H)$. The famous Erd\H{o}s-Ko-Rado (EKR)
theorem \cite{ekr} states that if $r\le \frac{n}{2}$ and a hypergraph $H
\subseteq \binom{[n]}{r}$ has more than $\binom{n-1}{r-1}$ edges,
then $\nu(H)
>1$.  This has been extended in more
than one way to pairs of hypergraphs. For example, in \cite{mt,
pyber} the following was proved:
\begin{theorem}\label{mt}
If $r \le \frac{n}{2}$, and $H_1,H_2\subseteq \binom{[n]}{r}$ satisfy
$|H_1||H_2|>\binom{n-1}{r-1}^2$ (in particular if
 $|H_i|>\binom{n-1}{r-1},~i=1,2$),
  then there
exist disjoint edges, $e_1 \in H_1,~e_2 \in H_2$.
\end{theorem}

In \cite{mt} this was also extended to hypergraphs of different
uniformities. In \cite{tok, tok13} a version of this result was
proved for $t$-intersecting pairs of hypergraphs, for large enough
$n$.

It is natural to try to extend the EKR theorem to more than two hypergraphs. The
relevant notion is that of ``rainbow matchings".

\begin{definition}
 Let $\cf=(F_i \mid~~1 \le i \le k)$ be a collection of hypergraphs. A
 choice of disjoint edges, one from each $F_i$, is called a {\em
 rainbow matching} for $\cf$.
 \end{definition}

\begin{notation}
 Given numbers $n,r,k$ satisfying $kr \le n$, let $f(n,r,k)$ be
the smallest number such that $\nu(H)\ge k$ for every $H \subseteq
\binom{[n]}{r}$ larger than $f(n,r,k)$.

Given any numbers $n,r,k$, let $g(n,r,k)$ be the smallest number
such that $\nu(H)\ge k$ for every $H \subseteq [n]^r$ larger than
$g(n,r,k)$. \end{notation}

\begin{lemma}
$g(n,r,k)=(k-1)n^{r-1}$.
\end{lemma}

\begin{proof}
To see that $g(n,r,k)\ge (k-1)n^{r-1}$ take $F$ to be  the set of
edges containing any of given $k-1$ vertices in the same side:
$|F|=(k-1)n^{r-1}$ and since the covering number is $k-1$ there is
no matching of size $k$. To show that $g(n,r,k)\le (k-1)n^{r-1}$,
let $F$ be any set of edges of size larger than $(k-1)n^{r-1}$. It
is easy to see that $[n]^r$ is the union of $n^{r-1}$ perfect
matchings $M_i$. By the pigeonhole principle $|F \cap M_i|>k-1$ for
some $i$, and since $F \cap M_i$ is a matching, it follows that
$\nu(F \cap M_i)\ge k$.
\end{proof}

The function $f(n,r,k)$ is hard to determine, see \cite{frankl,frr}
for estimates.

The motivation behind this paper is the following  conjecture:

\begin{conjecture}\label{rainbow}\hfill
\begin{enumerate}
\item
If $F_1, \ldots F_k$ are sub-hypergraphs of ~$[n]^r$, each of size
larger than $(k-1)n^{r-1}$, then there exists a choice of disjoint
edges $e_1 \in F_1, \ldots, e_k \in F_k$.
\item
If $F_1, \ldots ,F_k$ are sub-hypergraphs of $\binom{[n]}{r}$, all
larger than $f(n,r,k)$,  then there exists a choice of disjoint
edges $e_1 \in F_1, \ldots, e_k \in F_k$.
\end{enumerate}
\end{conjecture}

In \cite{ah} part (1) of this conjecture was proved  for $r\le 3$.
The case $k=2$ of part (1)  follows from Theorem \ref{main} below,
and was proved independently by Alon \cite{alon}, using a spectral
method. Theorem \ref{main} is more general: it determines the
maximal size of $|B(H)|$ for a sub-hypergraph $H$ of $[n]^r$, given
its size. In particular, it implies that if $|H|>n^{r-1}$ then
$|B(H)|<n^{r-1}$, which is the case $k=2$ of part (1).

In the second section we shall turn to the case $U=\binom{[n]}{r}$.
Daykin \cite{daykin} showed how the EKR theorem can be derived from
the Kruskal-Katona theorem. His proof also yields the case $r=2$ of
Conjecture \ref{rainbow}. The idea of the proof is that if $ |F|$ is
large then, by the Kruskal-Katona theorem,  the $r$-shadow of the
complements of the sets in $F$ is large, and hence the number of the
$r$-sets that meet all edges in $F$ is small. We extend this idea
and find the maximal size of $|B(H)|$ for a sub-hypergraph $H$ of
$\binom{[n]}{r}$, given the cardinality of $H$.

\subsection{A self-similar sequence}

Denote the sides of $[n]^r$ by  $V_1, \ldots, V_r$ (so, all $V_i$'s
are of size $n$). Choose one vertex $v_i$ from each $V_i$. Let
$\Psi_r$ be the set of (possibly empty) sequences $\sigma$ of length
at most $r-1$ consisting of $\wedge$'s and $\vee$'s. Let
$\Sigma_r=\Psi_r \cup \{\alpha, \omega\}$, where
  $\alpha=\alpha_r$ and $\omega=\omega_r$ are new elements. Note that $|\Sigma_r|=2^r+1$.
We define hypergraphs $F_r(\sigma)$ for all $\sigma \in \Sigma_r$,
as follows. Let $F_r(\alpha)=\emptyset$ and $F_r(\omega)= [n]^r$.
For a sequence $\sigma \in \Psi_r$  having length $m \ge 0$, and
whose $j$-th term is denoted by $\sigma_j$ ~($j \le m$), let:
$$F_r(\sigma)=\{e \in [n]^r \mid  v_1 \in e~ \sigma_1 (v_2 \in e~\sigma_2 (v_3 \in e \ldots \sigma_{m}(v_{m+1} \in e)\ldots )\}$$
\\
For example, $F_r(\emptyset)=\{e \in [n]^r \mid v_1 \in e\}$ and
$F_r(\wedge,\wedge,\vee)$ is the set of edges $e \in [n]^r$  satisfying:
$$v_1 \in e \wedge (v_2 \in e \wedge (v_3 \in e \vee(v_4 \in e)))$$

Let $f_r(\sigma)=|F_r(\sigma)|$.

\begin{lemma}\label{frsigma}
\hspace{.1cm}

If $\sigma \in \Psi_{r-1}$ then
\begin{enumerate}
\item
$f_r(\sigma)=nf_{r-1}(\sigma)$
\item
$f_r(\wedge,\sigma)=f_{r-1}(\sigma)$
\item
$f_r(\vee,\sigma)=n^{r-1}+(n-1)f_{r-1}(\sigma)$

\end{enumerate}
\end{lemma}

Part 1 is true since $F_r(\sigma)=F_{r-1}(\sigma)\times V_r$. Part 2 is true since an edge in $F_r(\wedge,\sigma)$ is obtained from an edge  $f\in F_{r-1}(\sigma)$, with indices shifted by 1, by adding $v_1$. Part 3 is true since
$F_r(\vee,\sigma)=\{v_1\}\times V_2 \times \ldots \times V_r \cup (V_1 \setminus \{v_1\}) \times F_{r-1}(\sigma)$ (where, again, edges in $F_{r-1}(\sigma)$ have their indices shifted by 1).

Order $f_r(\sigma)$ by size: $$0=f_r(\alpha) < f_r(\sigma_1)<
f_r(\sigma_2)<\ldots <f_r(\sigma_{2^r})$$

Define $N(i)=N_r(i)$ as $f_r(\sigma_i)$ ~$(0 \le i \le 2^r)$.

\begin{example}
\hspace{.1cm}
\begin{enumerate}
\item
$N(0)=f_r(\alpha)=0$.
\item
$N(1)=f_r(\wedge, \wedge, \ldots, \wedge)$ ($r-1$ times), which is $1$.
\item
$N(2)=f_r(\wedge, \wedge, \ldots, \wedge)$ ($r-2$ times) which is $n$.
\item
$N(2^{r-1})=f_r(\emptyset)=n^{r-1}$.
\item
$N(2^r)=f_r(\omega)=n^r$.
\end{enumerate}
\end{example}

In accord we order $\Sigma_r$ as $\sigma(i)~(0\le i \le 2^r)$.  For example $\sigma(0)=\alpha,~\sigma(2^r)=\omega$.
We also define the inverse function, which we name ``$i$": if $\sigma(q)=\tau$, then $i(\tau)=q$.

Clearly, for every  $\beta, \gamma,\delta \in \Psi_r$ such that
$(\beta, \wedge, \gamma)$ and  $(\beta,\vee, \delta)$ belong to
$\Psi_r$

\begin{equation} \label{order} i((\beta, \wedge, \gamma))<i(\beta)<i((\beta,\vee, \delta)) \end{equation}

The elements of $\Psi_r$ can be viewed as the nodes of a binary tree,  the depth of a node being the length of the sequence (so the root, with depth $0$, is the empty sequence).
The order on $\Psi_r$, uniquely determined by (\ref{order}), is known as the ``in-order depth first search" on the  tree,
where $\wedge$ (``left") precedes   $\vee$ (``right").

This description of the order on $\Psi_r$ entails an explicit
formula for $\sigma(i)$. Represent $i\neq 0,2^r$ in binary form: $i
= 2^{k_0}+2^{k_1}+\ldots+2^{k_s}$, where $k_0>k_1>\ldots >k_s$. Then
$\sigma(i)$ is of length $r-k_s-1$, and it consists of $s$ symbols
of $\vee$ and $r-k_s-1-s$ symbols of $\wedge$. It starts with
$r-k_0-1$ (possibly zero) $\wedge$'s; if $s>0$ these are followed by
a $\vee$; this is followed by $k_0-k_1-1$ (possibly zero)
$\wedge$'s, and if $s>1$ this is followed by a $\vee$, followed by
$k_1-k_2-1$ $\wedge$'s, and so forth.

For example, $\sigma_6(13)=\sigma_6(2^3+2^2+2^0)=(\wedge,\wedge,\vee,\vee,\wedge)$.

The numbers $N(i)$ can also be written explicitly:
$$N(i)=\sum_{i \le s}n^{k_i}(n-1)^i$$

The explicit description of $\sigma(i)$ and the formula for $N(i)$  will not be used below, and hence  their proofs are omitted.

\begin{example}
The values of $N_3$ are:
$$0,~1,~n,~n+(n-1),~n+n(n-1)=n^2,~n^2+(n-1),~n^2+(n-1)n,~n^2+(n-1)(2n-1),~n^2+(n-1)n^2=n^3.$$
\end{example}

\begin{lemma}\label{fourparts}
\end{lemma}

{\em \begin{enumerate}
\item
For $i \le 2^{r-1}$ we have $N_r(i)=N_{r-1}(i)$, namely the sequence $N_{r-1}(i)$ is  an initial segment of $N_r(i)$.

\item
$\sigma(2^p)=(\wedge,\wedge, \ldots,\wedge)$, a sequence of $r-p-1$ $\wedge$'s, and $N(2^p)=n^p$.

\item
For $i<2^p$ the sequences $\sigma(i)$ are of the form $(\sigma(2^p), \wedge, \beta)$ ($\beta$ being some sequence), and for
$2^p<i< 2^{p+1}$ the sequences $\sigma(i)$ are of the form $(\sigma(2^p), \vee, \beta)$.

\item
For $p\le r-1$ and $i \le 2^p$, we have
$$N(2^p+i)=N(2^p)+(n-1)N(i)=n^p+(n-1)N(i)$$

\end{enumerate}
}

Part 1 is true by part 2 of Lemma \ref{frsigma}, since $\sigma(1), \ldots, \sigma(2^{r-1}-1)$ all start with a $\wedge$. Parts 2 and 3 follow  from Equation (\ref{order}) and the remark following it. Part 4
follows from part 3 of Lemma \ref{frsigma}.

Part 4 says that the numbers $N(i)$ have a fractal-like pattern,
where each sequence $N_r$ is obtained from $N_{r-1}$ by adding on
its right an $n-1$-times magnified image of itself, the first
element of the right sequence being identified with the last element
of the left copy, both being equal to $n^{r-1}$. This entails:

\begin{lemma}\label{bcp}
If $b,c \le 2^p$ then $N(2^{p+1}+b)-N(2^p+c) =(n-1)(N(2^p+b)-N(c))$.
\end{lemma}

\subsection{Shifting}
{\em Shifting} is an operation on a hypergraph $H$, defined with
respect to a specific linear ordering ``$<$" on its vertices. For
$x<y$ in $V(H)$ define  $s_{xy}(e)=e \cup {x} \setminus \{y\}$ if $x
\not \in e$ and $y \in e$, provided
 $e \cup {x} \setminus \{y\} \not \in H$; otherwise let $s_{xy}(e)=e$.
  We also write $s_{xy}(H)=\{s_{xy}(e) \mid e \in H\}$. If
$s_{xy}(H)=H$ for every pair $x<y$ then $H$ is said to be {\em
shifted}.

Given an $r$-partite hypergraph $G$ with sides $M$ and $W$ together
with linear orders on each of its sides, an {\em $r$-partite
shifting} is a shifting $s_{xy}$ where $x$ and $y$ belong to the
same side. $G$ is said to be {\em $r$-partitely shifted} if
$s_{xy}(H)=H$ for all pairs $x<y$ that belong to the same side.

Given a collection $\ch=(H_i,~~i \in I)$ of hypergraphs, we write
$s_{xy}(\ch)$ for $(s_{xy}(H_i),~~i \in I)$.

As observed in \cite{erdosplus} (see also \cite{af}),  shifting does
not increase the matching number of a hypergraph. This can be
generalized to rainbow matchings (see, e.g., \cite{ah, hy}):

\begin{lemma}\label{shifting}
Let $\cf=(F_i \mid~~i\in I)$ be a collection of hypergraphs, sharing
the same linearly ordered ground set $V$, and let $x<y$ be elements
of $V$. If $s_{xy}(\cf)$ has a rainbow matching, then so does $\cf$.
\end{lemma}

\begin{proof}
Let $s_{xy}(e_i),~~i \in I$, be a rainbow matching for
$s_{xy}(\cf)$. There is at most one $i$ such that $x \in e_i$, say
$e_i=a \cup\{x\}$ (where $a$ is a set).


 If there
is no edge $e_s$ containing $y$, then replacing $e_i$ by $a
\cup\{y\}$ as a representative of $F_i$, leaving all other $e_s$ as
they are, results in a rainbow matching for $\cf$. If there is an
edge $e_s$ containing $y$, say $e_s=b \cup\{y\}$, then there exists
an edge $b \cup\{x\} \in F_s$ (otherwise the edge $e_s$ would have
been shifted to $b \cup\{x\}$). Replacing then $e_i$ by $a
\cup\{y\}$ and $e_s$ by $b \cup\{x\}$ results in a rainbow matching
for $\cf$.
\end{proof}

\subsection{The size of blocking hypergraphs}

For $\sigma \in \Psi_r$ we denote by  $\overline{\sigma}$ the sequence obtained by replacing each $\wedge$ by
a $\vee$ and vice versa. We also define $\overline{\alpha}=\omega$ and $\overline{\omega}=\alpha$. Clearly,
$i(\sigma)> i(\tau)$ if and only if $i(\overline{\sigma})<i(\overline{\tau})$, and hence we have:
\begin{equation}
i(\overline{\sigma}) = 2^r-i(\sigma)
\end{equation}

By De Morgan's law, we have:
\begin{lemma}
$B(F_r(\sigma))=F_r(\overline{\sigma})$.

\end{lemma}

\begin{lemma}\label{mainlemma}
 If~ $i \le j$ then $N(j+i)-N(j) \ge (n-1)N(i)$.
\end{lemma}

\begin{proof}
By induction on $i+j$. Assume that the lemma is true for all $i',j'$ whose sum is less than $i+j$, and let $s<j$. By the induction hypothesis:

\begin{equation}\label{weaker}
 N(s+i)\ge \max(N(i)+(n-1)N(s), N(s)+(n-1)N(i))\ge N(i)+N(s)
\end{equation}

Let $j=2^p+s$, where $s <2^p$. Assume first that $j+i\le 2^{p+1}$,
and write $j+i= 2^p+t$, where $t\le 2^p$. By part 4 of Lemma
\ref{fourparts} (the part saying that $N$-distances beyond $2^p$ are
$(n-1)$-magnified $N$-distances below $2^p$) we have
$N(j+i)-N(j)=(n-1)(N(t)-N(s))$. By (\ref{weaker}),  $N(t)-N(s)\ge
N(t-s)=N(i)$, and thus $N(j+i)-N(j) \ge (n-1)N(i)$.

Assume next that $j+i>2^{p+1}$ and write $j+i=2^{p+1}+w$. Then
$i=2^p+w-s$.

By the induction hypothesis we have $N(2^p+w)-N(s)\ge N(i)$. By
Lemma \ref{bcp} $N(2^{p+1})-N(2^p+s)=(n-1)(N(2^p)-N(s))$ and
$N(2^{p+1}+w)-N(2^{p+1})=(n-1)(N(2^p+w)-N(2^p))$. Adding the last
two equalities gives $N(j+i)-N(j)=(n-1)(N(2^p+w)-N(s))$, and since
by (\ref{weaker}) $N(2^p+w)-N(s)\ge N(i)$, we are done.

\end{proof}

A converse inequality is also true, namely for every $k>1$ it is true that:
\begin{equation} \label{knuth}
N(k)= \max\{N(j)+(n-1)N(i) \mid j+i=k,~i\le j\}
\end{equation}

\begin{proof}
Let $p$ be maximal such that $2^p <k$, and let $k=2^p+j$. By Lemma  \ref{frsigma} (4) $N(k)= N(i)+(n-1)N(j)$. Combining this with  Lemma \ref{mainlemma} proves the desired equality.
\end{proof}

In \cite{knuth} (\ref{knuth}) was used as a defining recursion rule for the sequence $N(i)$ (which appeared there in a different context.)

For a number $t \le n^r$ denote by $N^*(t)$  the number $q$ such that
$N(q-1)<t\le N(q)$. This is an approximate inverse of $N$.

\begin{theorem}\label{kkr}
 $b(t) =N(2^r-N^*(t))$ for every $t \le n^r$.
\end{theorem}

\begin{proof}
Let $F=F_r(\sigma(N^*(t))$. Then $|F|\ge t$, and since
$B(F)=F_r(\bar{\sigma})$, we have  $|B(F)|=N(2^r-N^*(t))$. This
 proves that $b(t)\ge N(2^r-N^*(t))$.
  To complete the proof we have to show that  for every $F \subseteq [n]^r$  of size $t$ we have $|B(F)| \le N(2^r-N^*(t))$. Write $q=N^*(t)$.
 We wish to show that $|B(F)|
\le N(2^r-q)$.
  We do this
by induction on $r$. The case $r=1$ is easy, so assume that we know
the result for $r-1$ and we wish to prove it for $r$.

Let $F^+=\{e \setminus V_r \mid v_r \in e \in F  \}$ and
$F^-=\{e \setminus V_r \mid  e \in F,~~v_r \not \in e\}$.

By Lemma \ref{shifting} we may assume that $F$ is $r$-partitely
shifted, which in particular entails  $F^- \subseteq F^+$. Let
$B^+=B_{r-1}(F^+)$ and $B^-=B_{r-1}(F^-)$, and let
$f^+=|F^+|,~f^-=|F^-|,~b^+=|B^+|,~b^-=|B^-|$. Then $b^- \le b^+$.
Clearly:
$$B(F)=(B^- \times \{v_r\}) \cup (B^+\times (V_r \setminus \{v_r\}))$$
and hence
\begin{equation}\label{bf}
|B(F)| = b^-+(n-1)b^+
\end{equation}

Let $i=N^*(f^-)$ and $j=N^*(f^+)$.  Also let
 $i'=N^*(b^+),~~j'=N^*(b^-)$.
By Lemma \ref{mainlemma} we have:

$$|F|\le f^+ +(n-1)f^- \le N(i+j)$$

and hence $i+j \ge q$. By the inductive hypothesis $j' \le
2^{r-1}-i$, and $i' \le 2^{r-1}-j$, and hence $i'+j' \le
2^r-(i+j)\le 2^r-q$. By (\ref{bf}) and Lemma \ref{mainlemma} ,
  $|B(F)|\le N(i'+j') \le N(2^r-q)$, as desired.

\end{proof}

Since $n^{r-1}=N(2^{r-1})=2^r-2^{r-1}$, the case $k=2$ of Conjecture \ref{rainbow}(1)  follows directly:

\begin{corollary}\label{largesizes}
A pair  $F_1,F_2$ of subsets of $[n]^r$ satisfying $|F_1|>n^{r-1}$ and $|F_2|\ge n^{r-1}$ has a rainbow matching.
\end{corollary}

Here is a strengthening of this result:

\begin{theorem}\label{alon}
If $F_1, F_2 \subseteq [n]^r$ and $|F_1||F_2|>n^{2(r-1)}$ then the pair $(F_1,F_2)$ has a rainbow matching.
\end{theorem}

The proof will follow from:

\begin{lemma}
$N(a)N(b) \le N(ab)$.
\end{lemma}

\begin{proof}
By induction on $a+b$. The case $a+b=0$ is trivial. By (\ref{knuth})
$N(a)=N(c)+(n-1)N(d)$ for some $c\le d <a$ such that  $c+d=a$, and
$N(b)=N(e)+(n-1)N(f)$ for some $e\le f <b$ such that  $e+f=b$. Then
$$N(a)N(b)=N(c)N(e)+(n-1)[N(d)N(e)+N(c)N(f)]+(n-1)^2N(d)N(f)$$
Using the induction hypothesis, we get:
$$N(a)N(b)  \le N(ce)+(n-1)[N(d)N(e)+N(c)N(f)]+(n-1)^2N(df)$$
Using Lemma \ref{mainlemma} twice we get:
$$N(a)N(b)\le N(ce+cf)+(n-1)N(de+df)\le N(ce+cf+de+df)=N(ab).$$
\end{proof}

The lemma implies that $N(2^{r-1}-q)N(2^{r-1}+q)\le N(2^{2(r-1)})$ for every $q\le 2^{r-1}$, meaning that $tb(t) \le n^{2(r-1)}$ for every $t \le n^{r-1}$, which is another way of formulating Theorem \ref{alon}.

\begin{remark}
Theorem \ref{alon} was independently proved by Alon \cite{alon}. His
proof uses spectral methods, as used   also in \cite{efp, friedgut}.
He also proved the following $t$-intersecting version:
\begin{theorem}
For every $t$ there exists $n=n_0(t)$ such that for every $n>n_0(t)$
and every pair $F_1, F_2 \subseteq [n]^r$, if $|F_1||F_2| >
n^{2(r-t)}$ then there are $e_1 \in F_1$ and $e_2 \in F_2$ such that
$|e_1 \cap e_2| < t$.
\end{theorem}
\end{remark}

\section{Blockers in $\binom{[n]}{r}$}

\subsection{Sequences of $\vee$'s and $\wedge$'s and the sets they define}
Let $n$ be a positive integer. For a sequence $\sigma=(\sigma_1,\sigma_2, \ldots, \sigma_m)$ of $\wedge$'s and
$\vee$'s ($m<n$) let $T(\sigma)$ be
the set of subsets $e$ of $[n]$, satisfying $$1 \in e
~\sigma_1~ (2\in e ~\sigma_2~ (3 \in e \ldots \sigma_m~ ({m+1} \in
e))\ldots)$$

For a number $r \le n$ let $T_r(\sigma)= T(\sigma) \cap \binom{[n]}{r}$. Let also $t_r(\sigma)=|T_r(\sigma)|$ (this is the analogue of $f_r(\sigma)$ of the first section).

\begin{example}\hfill
\begin{enumerate}
\item
If $\sigma=(\vee,\wedge,\vee,\wedge)$ then $T(\sigma)= \{e \in [n]
 \mid 1 \in e~\vee ~(2 \in e ~\wedge~ (3 \in e ~\vee
~(4 \in e \wedge 5 \in e)))\}$.
\item  $T_r(\emptyset)=\{e \in \binom{[n]}{r} \mid 1 \in e\}$, and thus $t_r(\emptyset)=\binom{n-1}{r-1}$.
\item If $\sigma=\wedge^{r-1}$ (meaning that  $\sigma_i=\wedge$ for all $i<r$) then
$T_r(\sigma)=\{e e \in \binom{[n]}{r} \mid \{1,2,\ldots,r\}
\subseteq e\}=\{[r]\}$, meaning that $t_r(\sigma)=1$.
\end{enumerate}
\end{example}

For a positive integer $r$, let $\Upsilon_r$ be the set of sequences
$\sigma=(\sigma_1,\sigma_2, \ldots, \sigma_m)$ consisting of fewer
than $r$ symbols of $\wedge$ and fewer than $r$ symbols of $\vee$.
 Let $\Theta_r = \Upsilon_r \cup\{\alpha\} \cup \{\omega\}$,
where   $\alpha$ and $\omega$ are two new elements. Define
$T_r(\alpha)=\emptyset$ and $T_r(\omega)=\binom{[n]}{r}$.

\begin{lemma}\label{countingtheta} $|\Theta_r| =\binom{2r}{r}+1$.
\end{lemma}

\begin{proof} define a map from $\Upsilon_r \setminus \{\emptyset\}$ to the
set of sequences of $r$ symbols $\wedge$ and $r$ symbols $\vee$, in
which $\sigma$ goes to a sequence $\psi(\sigma)$ obtained by
appending to it at its end a sequence of the form $\wedge \wedge
\ldots \wedge \vee \vee \ldots \vee$ or $ \vee \vee \ldots \vee
\wedge \wedge \ldots \wedge$, in which the first symbol is the
opposite of the last symbol of $\sigma$.  Clearly, $\sigma$ is
reconstructible from $\psi(\sigma)$, since the last symbol of
$\sigma$ is recognizable  - it is the first symbol, going from right
to left, in the third stretch of identical symbols in
$\psi(\sigma)$.
 The two sequences $ \vee \vee \ldots \vee \wedge
\wedge \ldots \wedge$ and  $\wedge \wedge\ldots  \wedge \vee \vee
\ldots  \vee $ are missing from the image, and remembering that
$\emptyset \in \Upsilon_r$ this proves that
$|\Upsilon_r|=\binom{2r}{r}-1$. \end{proof}

We now wish to order $\Theta_r$. For this purpose we extend every
sequence in $\Upsilon_r$ by appending a symbol $*$ at its end, and then
ordering $\Upsilon_r$ lexicographically, with the convention
 $\wedge <* < \vee$ (the ``*" is then discarded). We also define $\alpha$ to be the minimal
 element and $\omega$ to be the largest element of $\Theta_r$.

\subsection{The sequence $M_r(i)$}
Write $m=\binom{2r}{r}$. Let $\sigma_0=\alpha < \sigma_1 <
\sigma_2 < \ldots <\omega=\sigma_m$ be the order  defined above on
$\Theta_r$, and let $M(i)=M_r(i)= t_r({\sigma_i})$.

\begin{observation}
The sequence $M(i)$ is strictly ascending.
\end{observation}

 Here is for example the sequence for $r=3$ and general $n$:

$0,1,2, 3, n-2, n-1,n,2n-5,2n-4,3n-9, \binom{n-1}{2},
\binom{n-1}{2}+1, \binom{n-1}{2}+2, \binom{n-1}{2}+n-3,
\binom{n-1}{2}+n-2,
\binom{n-1}{2}+2n-7,\binom{n-1}{2}+\binom{n-2}{2},
\binom{n-1}{2}+\binom{n-2}{2}+1, \binom{n-1}{2}+\binom{n-2}{2}+n-4,
\binom{n-1}{2}+\binom{n-2}{2}+\binom{n-3}{2}, \binom{n}{3}$.

This sequence does not seem to behave as nicely as the sequence
$N(i)$ from the first section, but like the sequence $N(i)$ it has
landmarks.

\begin{theorem}\hfill
\begin{enumerate}
\item $\sigma(\binom{2r-i}{r})=\wedge^{i-1}$.
\item
$\sigma(\binom{2r}{r}-\binom{2r-i}{r-i})=\vee^{i-1}$.
\item
$M(\binom{2r-i}{r})=\binom{n-i}{r-i}$.
\item
$M(\binom{2r}{r}-\binom{2r-i}{r-i})=\binom{n-1}{r-1}+\binom{n-2}{r-1}+\ldots
+\binom{n-i}{r-1}$.
\end{enumerate}

\end{theorem}

\begin{proof}
Part (1): the sequences preceding $\wedge^{i-1}$ are those that
start with $\wedge^{i}$. Using the same idea as in the proof of
Lemma \ref{countingtheta}, we define a map between the set of the
sequences $\sigma$ preceding $\wedge^{i-1}$ and the set of sequences
of $r$ symbols of $\vee$ and $r-i$ symbols of $\wedge$: we complete
$\sigma$ to a sequence of  $r$ symbols $\vee$ and $r$ symbols
$\wedge$ by appending to $\sigma$ at its end a sequence $\vee \vee
\ldots \vee \wedge \wedge \ldots \wedge$ or $ \wedge \wedge \ldots
\wedge \vee \vee \ldots \vee$, where the first symbol of the
appended sequence is the opposite of the last symbol of $\sigma$.
The only sequence that is not in the image of this map is $\wedge^r
\vee^r$, and hence the number of sequences preceding $\wedge^{i-1}$
is $\binom{2r-i}{r}$-1.

Part (2) follows by symmetry. Parts (3) and (4) follow by simple counting.
\end{proof}

\subsection{Calculating $b(t)$ for  $t \le \binom{n}{r}$}

For $\sigma \in \Upsilon_r$  denote by  $\overline{\sigma}$ the sequence
obtained from $\sigma$ by replacing each $\wedge$ by a $\vee$ and vice versa. Also
define $\overline{\alpha}=\omega$ and $\overline{\omega}=\alpha$.
By De Morgan's law, we have:
\begin{lemma}\label{reciproc}
$B(T_r(\sigma))=T_r(\overline{\sigma})$.
\end{lemma}

 The main result of this section is:

\begin{theorem}\label{bisa}
For every number $0\le t \le \binom{n}{r}$  there exists $0 \le i \le \binom{2r}{r}$ such that $b(t)=M(i)$.
\end{theorem}

The proof uses an already mentioned idea of Daykin \cite{daykin},
who gave a proof of the EKR theorem using the Kruskal-Katona
theorem.

For a hypergraph $F$ and a number $r$, the {\em $r$-shadow} of $F$, denoted by $S_r(F)$, is
$\bigcup_{f \in F} \binom{f}{r}$.
 A hypergraph $F$ of uniformity $k$ is
said to be in ``cascade form'' if there exist sets $B_0=[n]
\supseteq
 B_1 \supsetneqq \ldots  \supsetneqq B_{s+1}$ and
elements $x_i \in B_{i-1}\setminus B_i ~( 1 \le i \le s)$, such that

$$F=\binom{B_1}{k} \cup x_1*\binom{B_2}{k-1}\cup
x_1*x_2*\binom{B_3}{k-2}\cup \ldots \cup x_1*x_2*\ldots
*x_s*\binom{B_{s+1}}{k-s}$$

Here ``*'' stands for the join operation, meaning that $x*H=\{h \cup \{x\} \mid h \in H\}$.

\begin{theorem}\label{kk}\cite{kruskal, katona}
Given numbers  $m, n$ and $r \le k$, the minimum of \ $|S_r(H)|$
over all $H \subseteq \binom{H}{k}$ is attained at a hypergraph $H$
having cascade form.
\end{theorem}

{\em Proof of Theorem  \ref{bisa}}  We have to show that  there
exists $\beta \in \Upsilon_r$ satisfying the following condition:
the maximum of $|B(H)|$ over all hypergraphs $H \subseteq
\binom{n}{r}$ of cardinality   $t$ is attained at a hypergraph $H$
for which $B(H)=T_r(\beta)$ for some sequence $\beta \in
\Upsilon_r$.

Clearly, $B(H)=S_r(\bar{H})^c$, where $\bar{H}$
is the set of complements of edges in $H$, and $S_r(\bar{H})^c$
denotes the set of all edges of size $r$ that do not belong to
$S_r(\bar{H})$. By Theorem \ref{kk} the maximal value of $|B(H)|$
over all $H \subseteq \binom{n}{r}$ is attained at a hypergraph $H$
for which $\bar{H}$ has cascade form. Let this form be

\begin{equation} \label{cascade}
\bar{H}=\binom{B_1}{n-r} \cup x_1*\binom{B_2}{n-r-1}\cup
x_1*x_2*\binom{B_3}{n-r-2}\cup \ldots \cup x_1*x_2*\ldots
*x_{s}*\binom{B_{s+1}}{n-r-s}
\end{equation}

Here possibly $s=0$. As above, we define $B_0=[n]$. For each $0 \le
i \le s$ let $B_i\setminus (B_{i+1}\cup\{x_i\})=\{z^i_1, \ldots,
z^i_{t_i}\}$, where $t_i=|B_i\setminus B_{i+1}|-1$ (Here possibly
$t_i=0$).

 \begin{assertion}  $B(H)=T(\theta)$, where
 $$
\theta=z^0_1 \vee (z^0_2 \ldots \vee (z^0_{t_0} \vee (x_1 \wedge
(z^1_1 \vee (z^1_2 \vee \ldots \vee(z^1_{t_1} \vee (x_2 \wedge
(z^2_1 \vee (z^2_2 \vee \ldots \vee(z^2_{t_2}\ldots
$$

if  $B_1 \neq [n]$ and $\theta=\alpha$
 if $B_1=[n]$.
\end{assertion}

To prove the assertion,  we have to show that  a set $e$ of size $r$
belongs to $S_r(\bar{H})^c$ if and only if it satisfies the
conditions imposed by $\theta$. If $e$ contains one of $z^0_1, z^0_2
\ldots ,z^0_{t_0}$ then it does not belong to $S_r(\bar{H})$ because
edges in $\bar{H}$ are contained in $\{x_1\}\cup B_0$. If $e$ does
not contain any of these vertices, it may still belong to
$S_r(\bar{H})^c$, if it contains $x_1$. In such a case if $e$ also
contains  none of  $z^1_1, z^1_2 \ldots ,z^1_{t_1}, x_2$ then it
belongs to $S_r(\bar{H})$. So, we may assume that $e$ contains one
of these vertices or it contains $x_2$ together with $x_1$, and so
on. This completes the proof of the assertion.

Next note that since $e$ is of size $r$, it suffices to stop just after $x_r$,
and obtain a condition that is satisfied by $e$ if and only if $e\in T(\theta)$. For example, for $r=2$ a set of size $2$ satisfies the condition
$$x_1 \wedge (z^1_1\vee(x_2 \wedge (z^2_1\vee x_3)))$$

  if and only if it satisfies the condition

$$x_1 \wedge (z^1_1\vee x_2)$$

Let $\beta$ be the formula obtained by truncating $\theta$ after
 $x_r$, if indeed $x_r$ appears, and let $\beta=\theta$ otherwise.

Note also that the number of $\vee$'s in $\theta$ is equal to the
number of $z^j_i$'s in $\theta$. The assumption is that the set
$\binom{B_{s+1}}{n-r-s}$ appearing in \eqref{cascade} is non-empty,
which implies that $|B_{s+1}| \ge n-r-s$. This is easily seen to
imply that the number of $z^j_i$'s is at most $r$.  Thus $\beta \in
\Upsilon_r$, which completes the proof of Theorem \ref{bisa}.
\\

We can now achieve our aim - the calculation of $b(t)$ for every $t
\le \binom{n}{r}$.

\begin{theorem}
If $M(i-1)<t\le M(i)$ then $b(t)=M({\binom{2r}{r}-i})$.
\end{theorem}

\begin{proof}
By Lemma \ref{reciproc}  $b(M(j))=M({\binom{2r}{r}-j})$ for all $0
\le j \le \binom{2r}{r}$. Since $b(c)\le b(d)$ whenever $c \ge d$,
this implies that $M({\binom{2r}{r}-i}) \le   b(t) \le
M({\binom{2r}{r}-i+1})$, and by Theorem \ref{bisa} it follows that
either $b(t) = M({\binom{2r}{r}-i+1})$ or $b(t) =
M({\binom{2r}{r}-i})$. By the definition of the function $b$ we have
$b(b(t)) \ge t$, and hence if $b(t) = M({\binom{2r}{r}-i+1})$ then
$t \le b(M({\binom{2r}{r}-i+1})=M(i-1)$, contradicting the
assumption of the theorem. Thus $b(t) = M({\binom{2r}{r}-i})$.
\end{proof}

\end{document}